\documentclass[11pt]{article}
\usepackage{amsmath,amsthm}
\usepackage{xcolor}
\usepackage[margin=1in]{geometry}

\usepackage{float}
\usepackage{graphicx}
\usepackage{amsfonts,amssymb}
\usepackage{makecell}
\usepackage{enumitem}
\usepackage{amsmath,amsthm,amssymb,amsfonts,epsfig,graphicx}
\usepackage[linesnumbered,ruled,vlined]{algorithm2e}

\newtheorem{theorem}{Theorem}[section]

\newtheorem{proposition}[theorem]{Proposition}
\newtheorem{lemma}[theorem]{Lemma}
\newtheorem{conjecture}[theorem]{Conjecture}
\newtheorem{corollary}[theorem]{Corollary}
\newtheorem{definition}[theorem]{Definition}

\newcounter{note}[section]

\author{
Nikhil Bansal\thanks{University of Michigan, Ann Arbor. {\small bansal@gmail.com}. Supported in part by the NWO VICI grant 639.023.812.} \\
\and
David G. Harris\thanks{
University of Maryland, College Park.
{\small davidgharris29@gmail.com}
}}
\date{}
\title{Some remarks on hypergraph matching and the F\"{u}redi-Kahn-Seymour conjecture}
\begin{document}

	\maketitle
\begin{abstract}
	A classic conjecture of F\"{u}redi, Kahn and Seymour (1993) states that any hypergraph with non-negative edge weights $w(e)$ has a matching $M$ such that $\sum_{e \in M} (|e|-1+1/|e|)\, w(e) \geq w^*$, where $w^*$ is the value of an optimum fractional matching. We show the conjecture is true for rank-3 hypergraphs, and is achieved by a natural iterated rounding algorithm. While the general conjecture remains open, we give several new improved bounds. In particular, we show that the iterated rounding algorithm gives 
	$\sum_{e \in M} (|e|-\delta(e))\, w(e) \geq w^*$, where $\delta(e) = |e|/(|e|^2+|e|-1)$, improving upon the baseline guarantee of $\sum_{e  \in M} |e|\,w(e) \geq w^*$.  
	\end{abstract}
		
	\section{Introduction}
	Let $H = (V,E)$ be a hypergraph  with vertex set $V$ and edge set $E$, where each edge $e \in E$ is a subset of $V$; for simplicity, we assume throughout that $\emptyset \notin E$. We define the \emph{rank} of $H$ to be the largest size of any edge in $E$. We define $E_k$ to be the set of edges of cardinality $k$; we call these \emph{$k$-edges} for brevity.  	For a vertex $v$, we define the edge-neighborhood $N(v) = \{ e : v \in e \}$ to be the edges containing $v$, and $N_k(v) = N(v) \cap E_k$ to be the $k$-edges containing $v$. Likewise, for an edge $e$, we define the (exclusive) edge-neighborhood $N(e) =  \{f: e \neq f,  e \cap f \neq \emptyset \} = \bigcup_{v \in e} N(v) \setminus \{e \}$ and  $N_k(e) = N(e) \cap E_k$.
		
	A matching $M$ of $H$ is a collection of pairwise disjoint edges; equivalently, it is an independent set of the line graph of $H$. 	A fractional matching of $H$ is a function $x: E \rightarrow [0,1]$ satisfying the condition
	 \begin{equation} \sum_{e \in N(v)} x(e) \leq 1  \qquad \text{ for all vertices } v \in V. \label{eq:matching}
	\end{equation} 
	In other words, viewing $x(e)$ as the fractional extent to which an edge is picked, the total fraction of edges containing any vertex $v$ is at most $1$. Clearly, if $x$ takes on integral values, then it corresponds to a matching. We say that $x$ is a \emph{basic} fractional matching if it is an extreme point of the fractional matching polytope given by the constraints \eqref{eq:matching}.
	
	Given a non-negative weight function $w: E \rightarrow \mathbb R_{\geq 0}$ on the edges, we define the \emph{weight} of a fractional matching $x$ by \[
	w(x) = \sum_{e \in E} w(e) x(e).\]
	Given the fundamental role of matchings in combinatorics, optimization and computer science, there has been quite some interest in understanding the integrality gap between the weight of the integral and fractional matchings.
	
	\paragraph{The FKS Conjecture.} With the aim of exactly pinning down the integrality gap, F{\"u}redi, Kahn \& Seymour conjectured a fine-grained relation between integral and fractional matchings \cite{fks}. This is  particularly powerful for settings where the edges have different sizes. There are two formulations of the  conjecture. The original version was stated in the primal form below. As shown in \cite{anegg}, the dual version is  equivalent to the primal version, both algorithmically and combinatorially.
	
	\begin{conjecture}[FKS conjecture \cite{fks}]
	\label{fks-conj}
	Let $H$ be a hypergraph with fractional matching $x$.
	
	(Primal) For any weight	function $w: E \rightarrow \mathbb R_{\geq 0}$, there exists a matching $M$ with
	\begin{equation}
	    \label{eq:primal-conj}
	    	\sum_{e \in M} (|e| - 1 + 1/|e|) w(e) \geq w(x).
	\end{equation}
\indent	(Dual) There is a probability distribution $\Omega$ over matchings of $H$, such that when $M$ is drawn from $\Omega$, every edge $e \in E$ satisfies
	\begin{equation}
	\label{fks:dual}
	    	\Pr(e \in M) \geq \frac{x(e)}{|e| - 1 + 1/|e|}.
	\end{equation}
	\end{conjecture}
	
	This conjecture was shown for a number of special cases in \cite{fks}, including \emph{uniform hypergraphs} where all the edges have the same size.  The factor $k-1+1/k$ for edges of size $k$ is optimal for infinitely many values of $k$ (it is tight for projective planes). 
	
	The fractional matching LP can easily be solved efficiently, while finding or approximating a maximum-weight matching is intractable even for rank-3 hypergraphs \cite{kahn}. Thus, the FKS conjecture is closely related to approximation algorithms for matchings. A number of recent papers \cite{brubach,anegg} have shown weakened versions of Conjecture~\ref{fks-conj}; most recently, \cite{anegg} showed the following:
	\begin{proposition}[\cite{anegg}]
	\label{anegg-result}	
	For a hypergraph $H$ with fractional matching $x$, there is a probability distribution $\Omega$ over matchings of $H$, such that when $M$ is drawn from $\Omega$, every edge $e \in E$ satisfies
\[
	\Pr(e \in M) \geq \frac{x(e)}{|e| - (|e|-1) x(e)}. 
\]
		\end{proposition}
	As $x(e)$ can be arbitrarily small for a particular edge, the denominator can be arbitrarily close to $|e|$, and in particular this gives the following.	
		\begin{corollary}[\cite{anegg}]
		\label{anegg-result2}
For a hypergraph $H = (V,E)$ with fractional matching $x$ and weight function $w: E \rightarrow \mathbb R_{\geq 0}$, there is a matching $M$ with
	\[
	\sum_{e \in M}  |e|  w(e) \geq w(x).\]
	\end{corollary}
	
To better situate the FKS conjecture, let us note a more general result for independent sets in graphs, based on constructions in \cite{anegg,bruggmann}.
\begin{theorem}
\label{turan-thm}
Let $G = (V,E)$ be a graph and $\lambda: V \rightarrow \mathbb R_{\geq 0}$ be a non-negative function. There is a probability distribution $\Omega$ over independent sets $I$ of $G$, such that when $I$ is drawn from $\Omega$, every vertex $v \in V$ satisfies
\[
\Pr(v \in I) \geq \frac{\lambda(v)}{\lambda(v) + \sum_{vu \in E} \lambda(u)}.\]
\end{theorem}
\begin{proof}
For each vertex $v$, draw an independent Exponential random variable $X_v$ with rate $\lambda(v)$. Place $v$ into $I$ if $X_v < X_u$ for all neighbors $u$ of $v$. The claim follows directly from the standard properties of the minimum of exponential random variables.
\end{proof}
Setting $\lambda(v) = 1$ for all $v$ in Theorem~\ref{turan-thm} gives one version of the well-known Caro-Wei theorem where each vertex $v$ goes into the independent set with probability $1/(\text{deg}(v) + 1)$.  Indeed, Theorem~\ref{turan-thm} should be thought of as a fractional version of the Caro-Wei theorem.  To see how Theorem~\ref{turan-thm} implies Proposition~\ref{anegg-result}, note that for any edge $e$ of $H$, the sum of $x(e)$ over neighboring edges $f$ is given by
\[
\sum_{f  \in N(e)} x(f) \leq \sum_{v \in e} \sum_{f \in N(v) \setminus \{e \} } x(f) \leq |e| ( 1 - x(e)).
\]
Thus,  applying Theorem~\ref{turan-thm} to the line graph $G$  of $H$ with $\lambda(e) = x(e)$, we get
\begin{equation}
    \label{eq:basic-bd}
    \Pr(e \in M) \geq \frac{x(e)}{x(e) + \sum_{f \in N(e)} x(f)} \geq \frac{x(e)}{x(e) + |e| (1 - x(e))}  = \frac{x(e)}{|e| - (|e|-1) x(e)}.
\end{equation}
In this sense, the results of Proposition~\ref{anegg-result} and Corollary~\ref{anegg-result2} are essentially only using crude degree statistics of the line graph of $H$. The FKS  conjecture asks whether it is possible to take advantage of underlying matching structure to improve over this bound.

For example, suppose that (wishfully speaking) we were guaranteed that $x(e) \in \{0 \} \cup [ 1/|e|, 1]$ for all $e$. In this case, \eqref{eq:basic-bd} would imply the bound in \eqref{fks:dual} and hence the FKS conjecture. While such a property is too much to hope for in general, it suffices  to have a lower bound on $x(e)$ in a certain average sense. As we discuss shortly, this argument was used already in \cite{fks} to show the FKS conjecture for uniform hypergraphs.

\section{Iterative rounding and our results}
Chan \& Lau \cite{chan} described an iterative rounding procedure to turn the existential bounds of \cite{fks} for uniform hypergraphs into an effective algorithm. We will extend this approach to cover non-uniform hypergraphs  (see Proposition \ref{approx-prop} below for the formal statement). Roughly speaking, this algorithm is based on the ``local-ratio'' technique \cite{BY04}. Given an instance $H$ the algorithm solves the fractional matching LP and considers a suitable edge $e$. It modifies the weight of neighbors of $e$  and removes $e$ to obtain a smaller hypergraph $H'$ and recursively applies the algorithm to $H'$. 

Consider a hypergraph $H=(V,E)$, where in addition each edge has a weight $w(e)$ as well as an associated ``discount factor'' $g(e) \in (0,1]$. 
For concreteness, one can think of $g(e) = 1/|e|$. We can use the following \textsc{FindMatching} algorithm to obtain a high-weight matching for $H$.

\begin{algorithm}[H]
\caption{\sloppy {$\textsc{FindMatching}(H, g, w)$}}
 \textbf{if} $H$ has no edges, \textbf{then} return $\emptyset$ and terminate. \\
\textbf{if} {there is some edge $e \in H$ with $w(e) \leq 0$} \textbf{then} return ${\textsc{FindMatching}(H \setminus \{e \}, g, w)}$ \\
 Solve the fractional matching LP to get a basic fractional matching $x$ maximizing $w(x)$. \\
\If{there is some edge $e \in H$ with $\sum_{f \in N(e)} g(f) x(f) \leq 1 - g(e) x(e)$} {
Choose an arbitrary such edge $e$ \\

Define new weight function $w'$ by 
\[
w'(f) = \begin{cases}
w(f) - w(e) g(f) / g(e) & \text{if $f \in N(e)$} \\
w(f) & \text{if $f \notin N(e)$}
\end{cases}
\] \\

Set $M' \leftarrow {\textsc{FindMatching}(H\setminus\{e \}, g, w')}$ \\
\If { $M' \cap N(e) = \emptyset$ } {
  return $M = M' \cup \{e \}$
  } \Else {  
  return $M = M'$
}
}
Output ``ERROR'' and terminate
\end{algorithm}

Note that while the earlier algorithm \cite{chan} used identical discount factors everywhere, ours can be different for each edge.  Proposition \ref{approx-prop} below analyzes the quality of the solution returned by the algorithm, in terms of the discount factors.
The following definitions are critical for our  analysis.
\begin{definition}
A basic fractional matching $x$ is \emph{stuck for $g$} if it satisfies the condition
\begin{equation}
\label{cx1}
\forall e \in E \qquad \sum_{f \in N(e)} g(f) x(f) > 1 - g(e) x(e).
\end{equation}
\end{definition}
\begin{definition}
Discount factor $g$ is \emph{good for $H$} if no non-empty subgraph of $H$ has a basic fractional matching that is stuck for $g$.
\end{definition}

\begin{proposition}
\label{approx-prop}Algorithm $\textsc{FindMatching}(H,g, w)$ runs in polynomial time, and outputs either ERROR or a matching of $H$. In the former case, $g$ is not good for $H$. In the latter case, the matching $M$ satisfies
$$
\sum_{f \in M} w(f)/g(f) \geq w^*
$$
where $w^*$ is the maximum weight of any fractional matching of $H$.
\end{proposition}
\begin{proof} Each iteration of \textsc{FindMatching} clearly runs in polynomial time. Each subproblem has its edge count reduced by one, so there are at most $|E|$ iterations in total.  Since we only add $e$ to $M$ if $N(e) \cap M' = \emptyset$, the output $M$ is clearly a matching.  We now show the two properties by induction on $|E|$. The base case where $E = \emptyset$ are clear in both cases.

For the induction step, observe that if the algorithm reaches line 12, then the fractional matching $x$ is evidently stuck for $g$. Otherwise, suppose that the recursive call on hypergraph $H \setminus \{ e \}$ at line 2 or 7 terminates in ERROR. By induction hypothesis, there is a basic fractional matching $x'$ which is stuck for $g$ for some $H' = (V, E')$ with $E' \subseteq E \setminus \{e \}$; this shows the first claim.

The induction step for the second claim is clear if some edge $e$ has $w(e) \leq 0$. Otherwise, let $M' = \textsc{FindMatching}(H\setminus\{e \}, g, w')$ where $e$ is the edge chosen at line 5, and let $F = M' \cap N(e)$. We calculate:
\begin{equation}
\label{eq:loc-ratio1}
\sum_{f \in M'} w(f)/g(f) = \sum_{f \in M'} w'(f)/g(f) + \sum_{f \in F} \frac{g(f) w(e)/g(e)}{g(f)} =  \sum_{f \in M'} w'(f)/g(f) + |F| w(e)/g(e).
\end{equation}
The restriction of $x$ to $H \setminus \{e \}$ is also a fractional matching. So by the induction hypothesis and using the definition of $w'$, we have
\begin{align*}
\sum_{f \in M'} w'(f)/g(f) &\geq \max_{\text{fractional matchings $x'$}} w'(x') \geq \sum_{f \in E \setminus \{e \}} w'(f) x(f) \\
&  =  \sum_{f \in E} w(f) x(f) - w(e) \Bigl(  \sum_{f \in N(e)} g(f) x(f) / g(e) + x(e) \Bigr),
\end{align*}

The criterion for choosing $e$ at line 5 ensures that $\sum_{f \in N(e)} g(f) x(f) \leq 1 - g(e) x(e)$. Since $w(e) > 0$, substituting this bound into the above formula gives:
\[
\sum_{f \in M'} w'(f)/g(f) 
\geq   w(x) - w(e)/g(e).
\]
Together with \eqref{eq:loc-ratio1}, this implies
\[
\sum_{f \in M'} w(f)/g(f) \geq w(x) +  ( |F| - 1 )  w(e)  / g(e).
\]
Now, if $F = \emptyset$, then $M = M' \cup \{e \}$, and so
\[
\sum_{f \in M} w(f)/g(f) = w(e)/g(e) + \sum_{f \in M'} w(f)/g(f)  \geq w(e)/g(e) + w(x) - w(e)/g(e)  = w(x).
\]
Otherwise, if $|F| > 0$, then $M = M'$ and so
\[
\sum_{f \in M} w(f)/g(f) = \sum_{f \in M'} w(f)/g(f)  \geq  w(x) +   ( |F| - 1 )  w(e)/g(e)  \geq w(x).
\]
In either case, it holds that $\sum_{f \in M} w(f)/g(f) \geq w(x) = w^*$.
\end{proof}

\subsection{Our results}

Given Proposition~\ref{approx-prop}, we make the following conjecture which would, in particular, show the FKS conjecture:
\begin{conjecture}
\label{main-conj1}
If we set $g(e) = \frac{1}{|e| - 1 + 1/|e|}$ for all edges $e$ of $H$, then $g$ is good for $H$.
\end{conjecture}

It was already shown in \cite{chan} that Conjecture~\ref{main-conj1} holds for uniform hypergraphs. As some further evidence for Conjecture~\ref{main-conj1}, and the FKS conjecture, we show the following:
\begin{theorem}
\label{result1}
If $H$ has rank 3, then Conjecture~\ref{main-conj1} holds for $H$.
\end{theorem}
We prove Theorem~\ref{result1} in Section~\ref{res1sec}.  In Section~\ref{res1sec}, we also prove a number of bounds for what we call \emph{bi-uniform hypergraphs,} i.e.,~hypergraphs which have two possible edge sizes. We consider general hypergraphs in Section \ref{res2sec}, and while we are not able to show Conjecture~\ref{main-conj1}, we can show the following result which is strictly stronger than Corollary~\ref{anegg-result2}:
\begin{theorem}
\label{result2}
 Let $h: \mathbb N \rightarrow \mathbb R$ be any function satisfying the following conditions for all $k \geq 2$:
 \vspace{-0.1in}
\begin{enumerate}
\setlength{\itemsep}{0pt}
\item[(A1)] $h(k+1) \leq h(k)$
\item[(A2)] $0 \leq h(k) \leq 1/(k-1+1/k)$
\item[(A3)] $h(k+1) \leq 1 - (k-1)\, h(k) $
\end{enumerate}
The discount factor defined by $g(e) = h(|e|)$ for all edges $e$ is good for any hypergraph $H$.   In particular, there is a probability distribution over matchings $M$ with $\Pr(e \in M) \geq h(|e|)$ for all $e$.
\end{theorem}
We prove Theorem~\ref{result2} in Section~\ref{res2sec}.  Notice that $h(k)=1/k$ satisfies each of the conditions (A1), (A2), (A3) with some slack, which will allow us to obtain improved bounds of the form $h(k) = 1/(k-\delta(k))$ for some functions $\delta(k)>0$.  There is no single optimal choice for the function $h$, and we give some sample bounds in the following corollary. 
\begin{corollary}
\label{maincor1}
Let $H = (V,E)$ be a hypergraph.
\begin{enumerate}
\item If $H$ has rank $r$, the following is good for $H$:
\[
g(e) = h_r(|e|) \qquad \text{for } h_r(k) =  \frac{(-1)^{r-k} (k-2)!}{(r-2)! (r + 1/r - 1)} + \sum_{i=1}^{r-k} \frac{ (-1)^{i+1}  (k-2)! }{ (k-2 + i)!}.
\]  
    \item With no restriction on rank, the following is good for $H$:
    \[
    g(e) =  h_{\infty}(|e|) \qquad \text{for } h_{\infty}(k) = \sum_{i=1}^{\infty} \frac{ (-1)^{i+1} (k-2)! }{ (k-2+i)!} 
  \] 
  \item With no restriction on rank, the following is good for $H$:
	\[
   g(e) = \tilde h_{\infty}(|e|) \qquad \text{for } \tilde h_{\infty}(k) =  \frac{1}{k - \delta(k)}, \text{where $\delta(k) := \frac{k}{k^2+k-1}$}.
  \]
\end{enumerate}
\end{corollary}
The proof of Corollary~\ref{maincor1} appears in Appendix~\ref{maincor1app}.

In order to compare and illustrate these bounds, let us denote the function  corresponding to the FKS conjecture by
\[ h^*(k) = 1/(k-1+1/k)\]
Figure~\ref{tab} shows the baseline value $1/k$ in Corollary~\ref{anegg-result2}, the conjectured function $h^*(k)$, and the values $h_{\infty}(k), \tilde h_{\infty}(k)$ from Corollary~\ref{maincor1}. 
\begin{figure}[H]
\begin{center}
\begin{tabular}{|c||c|c|c|c|}
\hline
$k$ & $1/k$ & $h^*(k) $ & $h_{\infty}(k)$ & $ \gape{\tilde h_{\infty}(k)}$ \\ 
\hline
\hline
2 & $0.5000$ & $0.6667$ & $0.6321$ & $0.6250$ \\
\hline
3 & $0.3333$ & $0.4286$ & $0.3679$ & $0.3667$ \\
\hline
4 & $0.2500$ & $0.3077$ & $0.2642$ & $0.2639$ \\
\hline
5 & $0.2000$ & $0.2381$ & $0.2073$ & $0.2071$ \\
\hline
6 & $0.1667$ & $0.1935$ & $0.1709$ & $0.1708$ \\
\hline
7 & $0.1429$ & $0.1628$ & $0.1455$ & $0.1455$ \\
\hline
8 & $0.1250$ & $0.1404$ & $0.1268$ & $0.1268$ \\
\hline
9 & $0.1111$ & $0.1233$ & $0.1124$ & $0.1124$ \\
\hline
10 & $0.1000$ & $0.1099$ & $0.1009$ & $0.1009$ \\
\hline
20 & $0.0500$ & $0.0525$ & $0.0501$ & $0.0501$ \\
\hline
\end{tabular}
\caption{\label{tab}Table showing $k$ along with a few approximation factors from the different results. 
Terms are rounded to four decimal places for readability. }
\end{center}
\end{figure}
To help parse these expressions, we record a few observations. 
\begin{itemize}
\item We have $h_k(k) = h^*(k) = \frac{1}{k-1+1/k}$, matching the known bound for $k$-uniform hypergraphs.
\item As $h_r(k)$ is given by an alternating sum,  the values $h_{k+2i}(k)$ are decreasing in $i$, and the values $h_{k+1+2i}(k)$ are increasing in $i$. 
They both converge to $h_{\infty}(k)$.  For any fixed $k$, the least value $h_r(k)$ is attained for $r=k+1$ which equals $h_{k+1}(k) = k^2/(k^3-1)$, that is, 
\[
h_r(k) \geq k^2/(k^3-1) >  1/k.\]
\item A simple calculation shows that $\tilde{h}_\infty(k) = h_{k+3}(k) \leq h_{\infty}(k) \leq h^*(k)$. While $h_\infty(k)$ is always larger than $\tilde{h}_{\infty}(k)$, they are quite close asymptotically. More specifically,
\[
h^*(k) = k^{-1} + k^{-2} - O(k^{-4}), \qquad  h_\infty(k),  \tilde h_{\infty}(k) =  k^{-1} + k^{-3} - O(k^{-4}).
\]
\item The value $h_{\infty}(k)$ can be computed in closed form as $h_{\infty}(k) = (-1)^{k} ( D_{k-2} - (k-2)!/\rm{e} )$, where $D_n$ is the number of derangements on $n$ letters and $\rm{e} = 2.718...$ is the base of the natural logarithm. For example, $h_\infty(3)=1/\mathrm{e}, h_\infty(4) = 1 - 2/\mathrm{e}$, etc. 
\end{itemize}

\paragraph{Intuition.}
Before the formal proofs of Theorem~\ref{result1} and Theorem~\ref{result2}, let us provide a high-level overview. The idea  is to show that if the algorithm gets stuck for some fractional solution $x$, then $x$ cannot be a basic solution to the LP defining the matching. This involves combining the inequalities \eqref{cx1} suitably to obtain a contradiction to some property of a basic solution.

Let us start with the simpler case of $k$-uniform hypergraphs (where the FKS conjecture is already known) and show that the algorithm does not get stuck for $g(e) = h^*(k) = 1/(k-1+1/k)$. Consider some basic solution $x$ to the LP, and let $V'$ be the set of tight vertices $v$ with $\sum_{f \in N(v)} x(f) =1$, and let $E'$ be the set of edges with $x(e)>0$. As $x$ is a basic solution, we have $|E'| \leq |V'|$.
We will show that if we assume that the algorithm is stuck, then the inequalities 
\begin{equation}
    \label{eq:discuss-stuck}
h^*(k) \Bigl( x(e) + \sum_{f \in N(e)} x(f) \Bigr) > 1
\end{equation}
can be combined to derive $|V'| < |E'|$, a contradiction.

Indeed, as $x$ is feasible we have $\sum_{f \in N(v)} x(f)\leq 1$ for each $v$; summing over the vertices $v$ in an edge $e$ gives  $\sum_{f} |f\cap e| x(f) \leq k$, and thus $k x(e) + \sum_{f \in N(e)} x(f) \leq k$.
Together with \eqref{eq:discuss-stuck}, this gives
\[   (k-1)x(e) < k-1/h^*(k) = 1-1/k,\]
or equivalently $x(e) < 1/k$. On the other hand, as $H$ is $k$-uniform and the vertices in $V'$ are all tight, summing over $e \in E'$ gives \[
k\sum_{e \in E'} x(e) = \sum_{v \in V} \sum_{e \in N(v)} x(e) \geq \sum_{v \in V'} \sum_{e \in N(v)} x(e) = |V'|.
\]
Together this implies $|V'|\leq k\sum_{e \in E'} x(e) < |E'|$. This is the desired contradiction, which is essentially the proof given in \cite{fks}.

\paragraph{Non-uniform edges.}
The computations for non-uniform edges sizes get  more involved. 
To explain, let us divide the edges into two groups: the edges $E_k$ of size $k$, and edges $E_{>k}$ of size strictly larger than $k$. Suppose that we use discount factors given by $g(e) = h^*(k)$ for $k$-edges and $g(e) = 1/|e|$ (the baseline value) for the larger edges.  

If all the $k$-edges were ``self-interacting,'' that is, had no neighbors in $E_{>k}$, then we could use the argument above for uniform hypergraphs. On the other hand, suppose that an edge $e \in E_k$ has only neighbors outside $E_k$. Then, for this edge $e$, we could calculate
\[
\sum_{f \in N(e)} g(f) x(f) \leq \sum_{f \in N(e)} x(f)/|f| \leq \sum_{f \in N(e)} x(f)/(k+1) \leq \frac{k}{k+1} (1 - x(e)) \leq \frac{k}{k+1},
\]
so that the condition \eqref{cx1} forces $x(e)$ to be large. Consequently, $x$ is not stuck for $g$ (it can make progress on edge $e$).
 
In Theorem~\ref{result2}, we track the amount of self-interaction for the $k$-edges, showing that we always can find some $k$-edge where $x(e)$ is large enough not to get stuck.

In Theorem~\ref{result1}, we track this more carefully: we not only keep track of the self-interaction among $2$-edges, but we also keep track of the self-interaction among $3$-edges. Since the hypergraph has only two edge sizes, the amounts of these self-interactions can be related to each other. 

Before we show Theorems~\ref{result1} and \ref{result2}, we next record some simple observations and definitions to analyze \textsc{FindMatching} and stuck fractional matchings.

\subsection{Simple bounds for stuck fractional matchings}
For a  fractional matching $x$, we say that a vertex $v$ is \emph{tight} if $\sum_{e \ni v} x(e) = 1$. We say that $x$ is \emph{reduced} if $x(e) \in (0,1)$ for all edges $e$.  We define $B \subseteq V$ to be the set of all tight vertices.  For any set of edges $L \subseteq E$, we define $B(L) = \bigcup_{e \in L} e \cap B$. 

\begin{proposition}
\label{lbprop}
If $x$ is reduced, then for any edge-set $L \subseteq E$, there holds $|L| \leq |B(L)|$.
\end{proposition}
\begin{proof}
Let us consider the restriction of the fractional matching LP to the edges in $L$, i.e.,~all the values $x(e)$ for $e \notin L$ are held fixed. This LP can be described as:
\begin{equation}
\label{a1}
\forall v \in V \qquad \sum_{e \in L \cap N(v)} x(e) \leq 1 - \sum_{f\in N(v) \setminus L} x(f)
\end{equation}
where we emphasize that the terms on the RHS are viewed as fixed constants.

If $v \notin B$, then the constraint \eqref{a1} is not tight. If $v \in B \setminus B(L)$, then the LHS of \eqref{a1} is zero. Thus, when restricting our attention to the variables $x(e)$ for $e \in L$, the LP has at most one tight non-trivial constraint for each $v \in B(L)$. Since we are assuming that $x(e) \in (0,1)$ for all $e$, it has a fractional variable for each $e \in L$. Since $x$ is a basic LP solution, the number of fractional variables $|L|$ is at most the number of tight constraints $|B(L)|$.
\end{proof}

\begin{proposition}
\label{nosing1}
If a basic fractional matching $x$ is stuck for $g$, then $H$ has no singleton edges.
\end{proposition}
\begin{proof}
Suppose that $e = \{v \}$ is such an edge. On the other hand, since $x$ is a fractional matching, we have $\sum_{f \in N(v)} x(f) \leq 1$. Since $N(e) = N(v) \setminus \{e \}$, this implies
\[
\sum_{f  \in N(e)} g(f) x(f) = \sum_{f \in N(v) \setminus \{e \} } g(f) x(f) \leq \sum_{f \in N(v) \setminus \{e \} } x(f) \leq 1 - x(e) \leq 1 - g(e) x(e),
\]
a contradiction.
\end{proof}

\begin{proposition}
\label{reduce-prop}
If $g$ is not good for $H$, then there is some non-empty subgraph $H' = (V, E')$ with a reduced basic fractional matching $x$ which is stuck for $g$.
\end{proposition}
\begin{proof}
Suppose that $E' \subseteq E$ has minimal size such that some basic fractional matching $x$ on hypergraph $H' = (V,E')$ is stuck for $g$ and $E'$ is non-empty. We claim that $N(e) \neq \emptyset$ for all $e \in E'$. For, if $N(e) = \emptyset$, then edge $e$ would satisfy
\[
0 = \sum_{f \in N(e)} g(f) x(f) > 1 - g(e) x(e),
\]
which is a contradiction since the RHS is non-negative. 

So we may assume that $|E'| > 1$. We claim now that $x$ is reduced. For, suppose that $x(e) = 0$ for some edge $e$. If we define the non-empty edge set $E'' = E' \setminus \{e \}$, we see that every edge $e'' \in E''$ still satisfies
\[
\sum_{f \in E'' \cap N(e'')} g(f) x(f) = \sum_{f \in E' \cap N(e'')} g(f) x(f) > 1 - g(e) x(e'').
\]
Thus, the restriction of $x$ to $H'' = (V, E'')$ is also stuck. This contradicts minimality of $E'$.

Finally, suppose that $x(e) = 1$ for some edge $e$. Since $N(e) \neq \emptyset$, there must be some edge $e' \in N(e)$; but then necessarily $x(e') = 0$, which we have already ruled out.
\end{proof}

\section{Bounds for bi-uniform hypergraphs}
\label{res1sec}
Let us consider a hypergraph $H$ with a reduced basic fractional matching $x$  where all edges have sizes $k, \ell$ where $k < \ell$. We use discount factors given by $$
g(e) = \begin{cases} p & \text{for $|e| = k$} \\
q & \text{for $|e| = \ell$}
\end{cases}
$$
for values $p, q$. We may assume $p \geq q$ as $k < \ell$. Most of the steps in the argument hold more generally for arbitrary edge sizes; we later  specialize to $k=2, p = h^*(2) = 2/3, \ell=3, q= h^*(3) = 3/7$  to obtain Theorem~\ref{result1}. 

For any edge $e$ and any integer $i \in \{k, \ell \}$ we define \[
a_i(e) = \sum_{f \in N_i(e)} x(f). 
\]
We will adopt the following notation throughout this section: we use $e$ to denote a $k$-edge and $f$ to denote a $\ell$-edge. For a vertex $v$, let $x(v)$ denote the total $x$-value of edges containing $v$.  We begin with the following upper bounds on $x(e)$ and $x(f)$.
\begin{proposition}
\label{eq:bounds}
If $x$ is stuck for $g$, then all edges $ e \in E_k, f \in E_{\ell}$ satisfy
\[
x(e) < \frac{k p-1 - (p - q) a_{\ell}(e)}{(k-1)p} \qquad \text{and} \qquad x(f)  < \frac{\ell q -1 + (p - q) a_k(f)}{(\ell-1)q}.  
\]
\end{proposition}
\begin{proof}
The condition \eqref{cx1} for $x$ being stuck implies that 
\[ 1 - p a_k(e) - q a_\ell(e) < p x(e) \quad \forall e \in E_k
\qquad \text{and} \qquad 1 -  p a_k(f) - q a_\ell(f)  < q x(f) \quad \forall f \in E_{\ell}.  \]
As the total $x$-value for each vertex is at most $1$, we have the bounds:
\[ a_k(e) + a_\ell(e) \leq k (1-x(e)) \quad \forall e \in E_k     \qquad \text{and} \qquad  a_k(f) + a_\ell(f) \leq \ell(1-x(f)) \quad \forall f \in E_{\ell}. \]
Eliminating $a_\ell(f)$ from the inequalities for $f$, and eliminating $a_k(e)$ from the inequalities for $e$ gives the two claimed bounds.
\end{proof}

For an $\ell$-edge $f$, let $n_k(f) = | N_k(f) |$ denote the number of $k$-edges (possibly zero) incident to $f$. Then Proposition~\ref{eq:bounds} above implies the following  upper bound on $x(f)$ in terms of $n_k(f)$.
\begin{proposition}
\label{x-n}
If $x$ is stuck for $g$, then for any $f \in E_{\ell}$,  
\[ x(f) \Bigl( 1 + n_k(f) \frac{(p-q)^2}{(k-1) (\ell-1) p q} \Bigr) < \frac{\ell q-1}{(\ell-1) q} + n_k(f) \frac{ (kp -1) (p-q)}{(k-1)(\ell-1) p q}.  \] 
\end{proposition}
\begin{proof}
Consider any $k$-edge $e$ incident to $f$. Then by the bound on $x(e)$ in Proposition \ref{eq:bounds},
\[ x(e) <    \frac{k p-1 - (p - q) x(f)}{(k-1)p},\]
where we use that $a_\ell(e) \geq x(f)$ as $f$ is adjacent to $e$ (and that $p \geq q$). As this holds for every $e \in N_k(f)$, we get
\[ a_k(f) = \sum_{e \in N(f)} x(e) < n_k(f) \Bigl( \frac{k p-1 - (p - q) x(f)}{(k-1)p} \Bigr).\]

If we substitute this upper bound for $a_k(f)$ into the bound of Proposition~\ref{eq:bounds} for $x(f)$, we get 
\[
x(f) < \frac{\ell q -1}{(\ell - 1) q} + \frac{p - q}{ (\ell - 1) q} \cdot  n_k(f) \Bigl( \frac{k p-1 - (p - q) x(f)}{(k-1)p} \Bigr)
\]
which, upon rearranging, is equivalent to the stated inequality.
\end{proof}

Let us now consider the sum
\[ S = \sum_{e \in E_k} (k x(e) -1) + \sum_{f \in E_{\ell}} (\ell x(f) -1).\]
Our strategy is to show upper and lower bounds for $S$, under the assumption that $x$ is basic, reduced, and stuck for $g$. We then argue that these bounds are contradictory, and hence the desired matching exists. 
We show the lower bound and upper bound on $S$ in turn.

\paragraph{Lower bound.}  First, as $x(v) \leq 1$ for any vertex $v$, and \[
\sum_{e \in E_k} k x(e)  + \sum_{f \in E_{\ell}} \ell x(f) = \sum_v x(v),
\]
Proposition~\ref{lbprop} gives
\begin{equation}
\label{eq:ss1}
S = \sum_{v \in V} x(v)  - |E| \geq \sum_{v \in B} x(v) -  |E| = |B| - |E| \geq 0.
\end{equation}

\paragraph{Upper bound.} By the bound on $x(e)$ in Proposition \ref{eq:bounds}, we have
\[ S \leq \sum_{ e \in E_k} \Bigl( \Bigl(\frac{k^2 p-k}{(k-1)p}-1 \Bigr) -  \frac{k(p -q)}{(k-1) p} a_\ell(e) \Bigr) + \sum_{f \in E_{\ell}} (\ell x(f) -1).\] 
Writing \[\sum_{e \in E_k} a_\ell(e) = \sum_{e \in E_k}  \sum_{f \in N_{\ell}(e)} x(f) = \sum_{f \in E_{\ell}} x(f) \sum_{e \in N_k(f)} 1 = \sum_{f \in E_{\ell}} x(f) n_k(f), \]
this gives
\begin{equation}
\label{eq:s}
 S \leq \sum_{e \in E_k} \Bigl( \frac{k^2 p-k}{(k-1)p}-1 \Bigr) + \sum_{f \in E_{\ell}} \Bigl( \Bigl(\ell -n_k(f) \frac{k(p -q)}{(k-1) p} \Bigr) x(f) -1 \Bigr). 
\end{equation}

Putting together the bounds on $S$, we can now conclude our desired contradiction as follows.
\begin{lemma}
\label{concl-prop}
Let $T =  \frac{p (k-1) \ell}{(p - q) k}$. Suppose that $p \leq h^*(k)$ and for every integer $n \in \{0, \dots, \lfloor T \rfloor \}$ it holds that
\begin{equation}
\label{gfa}
\frac{p q (k-1) (\ell-1)+n
   (p-q)^2}{p (k-1) \ell - k n
   (p-q)} \geq \frac{p (k-1) (q \ell-1)+n (p-q)
   (p k-1)}{p  (k-1)}   
 \end{equation}

Then, there cannot be a reduced basic fractional matching $x$ which is stuck for $g$.
\end{lemma}
\begin{proof}
If we examine upper bound (\ref{eq:s}), we see that the term $\frac{k^2 p-k}{(k-1)p}-1$ is non-positive since $p \leq h^*(k)$. We also must have $S \geq 0$ from (\ref{eq:ss1}). Now, if $E_{\ell} = \emptyset$ (and there are no summands $f$), then in fact $H$ is a $k$-uniform hypergraph. We have already shown that $x$ cannot be stuck in this case since $p \leq h^*(k)$. So, there must be an edge $f \in E_{\ell}$ for which the corresponding summand in (\ref{eq:s}) is non-negative. For this edge $f$, let us denote $n = n_k(f)$, i.e.
\begin{equation}
\label{hppl1}
\Bigl( (\ell -n \frac{k(p -q)}{(k-1) p} \Bigr) x(f) -1  \geq 0.
\end{equation}
Since $x(f) \geq 0$, we clearly must have \[
 n \leq \frac{p (k-1) \ell}{(p - q) k} = T,\]
 and then (\ref{hppl1}) is equivalent to
 \[
x(f) \geq \frac{ p (k-1) }{p (k-1) \ell - (p - q) k n}.\]
Substituting this lower bound for $x(f)$ into the bound of Proposition~\ref{x-n}, we conclude
\begin{equation}
\label{at1}
\Bigl( \frac{ p (k-1) }{p (k-1) \ell - (p - q) k n}  \Bigr)  \Bigl( 1 + n \frac{(p-q)^2}{(k-1) (\ell-1) p q} \Bigr) < \frac{\ell q-1}{(\ell-1) q} + n \frac{ (k p -1) (p-q)}{(k-1)(\ell-1) p q}.
\end{equation}
With some algebraic simplifications,  this inequality is precisely what is ruled out by (\ref{gfa}) for the given value $n$.
\end{proof} 

\paragraph{Proof of Theorem~\ref{result1}.} Now, let $H=(V,E)$ be a rank-$3$ hypergraph. Assume for contradiction that Theorem~\ref{result1} fails for $H$. Then by Proposition~\ref{reduce-prop} there is a subgraph $H'$ of $H$ with a reduced basic fractional matching $x$ stuck for $g$. The subgraph $H'$ also has rank at most $3$, and by Proposition~\ref{nosing1} has no singleton edges.   Thus, $H'$ is a bi-uniform hypergraph with $k = 2, \ell = 3$.

Now consider Lemma~\ref{concl-prop} for our desired bounds for $p = h^*(k)  = 2/3$ and $q=h^*(\ell) = 3/7$. The condition $p \leq h^*(k)$ is obviously satisfied. Here,  $T =  4.2$, and we just need to verify (\ref{gfa}) at the values $n = 0, 1, 2, 3, 4$, which is easily done. This shows Theorem~\ref{result1}. 

 (Note that (\ref{gfa}) is false for some \emph{non-integer} values $n$ in the relevant range, in particular, is it false for $n \in (0,0.8)$.)
 
 \paragraph{Other edge sizes.} We cannot show Conjecture~\ref{main-conj1} for other values of $k, \ell$. Nonetheless, if we set $p = h^*(k)$, we can still obtain non-trivial bounds on $q$.  For instance, for $\ell = k+1$, we have the following crisp estimate:
\begin{theorem}
\label{thm:k+1}
If $\ell = k+1$, then no basic fractional matching $x$ can be stuck with $$
p = h^*(k) = \frac{1}{k-1+1/k}, \qquad q = \frac{1}{k+1/k}.
$$  

Note that this $q$ is quite close to $h^*(k+1)=\frac{1}{k+1/(k-1)}$. In particular,  
$$
h^*(k+1) - q = k^{-4}+ O(k^{-5}).
$$
\end{theorem}
\begin{proof}
Plugging the bounds for $p$ and $q$ into Lemma~\ref{concl-prop} gives $T= \frac{k^4 - 1}{k^2}$. It can be routinely checked that (\ref{gfa}) holds for all real numbers $n \in [0, T]$. For example, it is a rational inequality that can be verified using an algorithm for decidability of real-closed fields. We used Mathematica's {\tt Reduce} command to check this.
\end{proof}

The proof of Theorem~\ref{thm:k+1} did not use the integrality of $n$, which was critical in proving Theorem~\ref{result1}. For fixed values $k, \ell$, we can obtain better bounds compared to Theorem~\ref{thm:k+1}.  For example, for $k=3$, $\ell=4$, we can set  $p = h^*(3) = 3/7$ and numerically optimize $q$  to satisfy Lemma~\ref{concl-prop} (trying all possible integer values of $n$) to obtain $q = 0.30508$; by contrast, Theorem~\ref{thm:k+1} would give $p = h^*(3), q = 0.3$.  Similarly, for $k=4,\ell=5$, we can set $p = h^*(4) = 4/13$, and $q = 0.23656$; by contrast, Theorem \ref{thm:k+1} gives $p = h^*(4) $, and $q = 0.23529$.

\section{Proof of Theorem~\ref{result2}}
\label{res2sec}
Recall here that we are setting $g(e) = h(|e|)$ where $h$ satisfies the bounds (A1), (A2), (A3).  Throughout this section, we assume for contradiction that $H$ is a minimal counter-example, i.e., $H$ is non-empty and there is a basic fractional matching $x$ of $H$ which is stuck for $g$.

 In light of Proposition~\ref{reduce-prop}, we assume that $x$ is reduced. By Proposition~\ref{nosing1},  $H$ then has no singleton edges, and so let $k \geq 2$ be the size of the smallest edge in $E$. 

We define $p = h(k), q = h(k+1)$. By properties (A1), (A2), (A3), we have that
\begin{equation}
\label{r1}
q \leq 1 - (k-1)p, \qquad 0 \leq  q \leq p \leq \frac{1}{k-1+1/k}.
\end{equation}

Let us write $E_{>k}, N_{>k}(v), N_{>k}(e)$ to denote the respective sets of edges with strictly more than $k$ vertices. Let $V'$ denote the set of vertices $v$ with $N_k(v) \neq \emptyset$. We also define $B' = B(E_k) = B \cap V'$ and $n_k(v) = |N_k(v)| \geq 1$ for all $v \in V'$. For each vertex $v$ and edge $e$, define
\[
y(v) = \sum_{f \in N_{>k}(v)} x(f), \qquad y(e) = \sum_{v \in e} y(v) = \sum_{f\in N_{>k}(e)} |f\cap e|\, x(f),
\]
so that $y(v)$ is the $x$-value of the edges containing $v$ which have size strictly larger than $k$, and $y(e)$ is the sum of $y(v)$ over all $v\in e$.

Consider the sum \[
S = \sum_{e \in E_k} x(e).\]

Our strategy is to compute upper and lower bounds on $S$. We then combine them and argue that they cannot both hold simultaneously, and hence the counter-example $H$ is impossible. 

\paragraph{Lower bound.} This is straightforward.  As each edge in $E_k$ has cardinality exactly $k$ and all its vertices are in $V'$, we can use double-counting to get:
\begin{align}
\label{vvt1}
S &=  \frac{1}{k} \sum_{v \in V'} \sum_{e \in N_k(v)} x(e) \geq \frac{1}{k} \sum_{v \in B'} \sum_{e \in N_k(v)} x(e) = \sum_{v \in B'} (1 - y(v))/k.
\end{align}

\paragraph{Upper bound.} First, we claim that every edge $e \in E_k$ satisfies 
\begin{equation}
\label{vvt2}
x(e) < (k p - 1 - (p - q) y(e))/(p (k-1)).
\end{equation}
To show this,  note that since $x$ is stuck, the condition \eqref{cx1} and the fact that $g(e)=p$ implies
\begin{align*}
 1 - p x(e) & <  \sum_{f \in N(e)} g(f)\,x(f) 
             \leq  p \sum_{f \in N_k(e)} x(f) + q \sum_{f \in N_{>k}(e)} x(f)   \\
            & \leq  p  \sum_{f \in N_k(e)} |f\cap e|\, x(f) + q \sum_{f \in N_{>k}(e)} |f\cap e|\,x(f) \\
						& \leq  p  ( k-k x(e) - y(e) ) + qy(e),  
\end{align*}
where the second inequality uses that $q = h(k+1) \geq g(f)$ for all $f \in N_{>k}(e)$.
(\ref{vvt2}) then follows upon rearranging the above inequality. 

We then sum inequality (\ref{vvt2}) over the edges $e \in E_k$ to obtain:
\begin{align*}
S &< \sum_{e \in E_k} \frac{ k p - 1 - (p-q) y(e)}{p (k-1)} = \frac{ |E_k| (k p - 1) - (p-q) \sum_{e \in E_k} y(e)}{p (k-1)} \\
& = \frac{ |E_k| (k p - 1) - (p-q) \sum_{e \in E_k} \sum_{v \in e} y(v)}{p (k-1)} =  \frac{ |E_k| (k p - 1) - (p-q) \sum_{v \in V'}  y(v) n_k(v)}{p (k-1)}.
\end{align*}
By Proposition~\ref{lbprop} we have $|E_k| \leq |B'|$. Noting that $p-q \geq 0, p\geq 1/k$ and $B' \subseteq V'$, we thus have:
\begin{equation}
\label{ff3}
S <  \frac{ |B'| (k p - 1) - (p-q) \sum_{v \in B'} y(v) n_k(v)}{p (k-1)} = \sum_{v \in B'}  \frac{ k p - 1 - y(v) n_k(v) (p-q)}{p (k-1)}. 
\end{equation}

\paragraph{Putting them together.} We now combine the upper bound inequality (\ref{ff3}) with the lower bound inequality (\ref{vvt1}) to get:
\[
\sum_{v \in B'} (1 - y(v))/k \leq S <  \sum_{v \in B'}  \frac{ k p - 1 - y(v) n_k(v) (p-q)}{p (k-1)}. 
\]
In particular, one of the summands on the LHS must be smaller than the corresponding summand on the RHS. Thus, there must exist some vertex $v \in B'$ with
\begin{equation}
\label{ff4}
1 - y(v) < k \Bigl( \frac{ k p - 1 - y(v) n_k(v) (p-q)}{p (k-1)}  \Bigr).
\end{equation}
Since $v \in B'$, we have $n_k(v) \geq 1$. Also, the first condition in \eqref{r1} implies $p-q \geq kp-1 \geq 0$. Thus, the right side above is at most 
\begin{equation}
\label{ff5}
\bigl( 1 - y(v) \bigr)  \Bigl( \frac{ k (k p - 1)}{p (k-1)}  \Bigr).
\end{equation}
The second condition in \eqref{r1} is equivalent to $\frac{ k (k p - 1)}{p (k-1)}  \leq 1$; thus, (\ref{ff5}) is at most $1 - y(v)$. Thus, this is a contradiction to the existence of $H$, completing the proof of Theorem~\ref{result2}.

\section{Acknowledgments}
Thanks to Rico Zenklusen for some helpful explanations about their paper \cite{anegg}. Thanks to the journal reviewers for useful suggestions and revisions.

\appendix

\section{Proof of Corollary~\ref{maincor1}}
\label{maincor1app}
\begin{proposition}
\label{maincor11}
For any value $r$, let us extend the definition of function $h_r$ by setting
$$
h_r(k) =  \begin{cases}
\frac{(-1)^{r-k} (k-2)!}{(r-2)! (r + 1/r - 1)} + \sum_{i=1}^{r-k} \frac{ (-1)^{i+1}  (k-2)! }{ (k-2 + i)!} & \text{if $2 \leq k \leq r$} \\
0 & \text{if $k > r$} 
\end{cases}
$$

This function $h_r$ satisfies properties (A1)--(A3).
\end{proposition}
\begin{proof}
Let us fix $r$ and write $h = h_r$ for brevity. Properties (A1)--(A3) can be easily checked for $k \geq r$ (noting that $h(r) = \frac{1}{r - 1 + 1/r})$.  We claim that (A3) holds with equality $k = 2, \dots, r - 1$, i.e., 
\begin{equation}
\label{sa0}
h(k+1) = 1 - (k-1) h(k).
\end{equation} 
To verify this, we calculate
\begin{align*}
h(k+1) + h(k)(k-1) &= \frac{ (-1)^{r-(k+1)}( (k+1-2)!  }{(r-2)! (r + 1/r -1 )} + \sum_{i=1}^{r-k-1} \frac{ (-1)^{i+1} (k+1-2)! }{(k+1-2+i)!} \\
& \qquad \qquad + \frac{ (-1)^{r-k} (k-1) (k-2)!  }{(r-2)! (r + 1/r -1 )} + (k-1) \sum_{i=1}^{r-k} \frac{ (-1)^{i+1} (k-2)! }{(k-2+i)!} \\
&= \sum_{i=1}^{r-k-1} \frac{ (-1)^{i+1} (k+1-2)! }{(k+1-2+i)!} + (k-1) \sum_{i=1}^{r-k} \frac{ (-1)^{i+1} (k-2)! }{(k-2+i)!} \\
&= \sum_{i=1}^{r-k-1} \frac{ (-1)^{i+1} (k-1)! }{(k-1+i)!} + \sum_{j=0}^{r-k-1} \frac{ (-1)^{j} (k-1)! }{(k-1+j)!} \qquad \text{substituting $j = i-1$} \\
&= \sum_{i=1}^{r-k-1} \Bigl( \frac{(-1)^{i+1} (k-1)! + (-1)^{i}(k-1)!}{(k-1+i)!} \Bigr) + \frac{ (-1)^0 (k-1)! }{(k-1+0)!} = 1,
\end{align*}
where the second equality follows as the first and third terms in the first equality cancel.

We now turn to verifying property (A2) for $k = 2, \dots, r -1$. To do so, we show the stronger inequality for $k \leq r$: \begin{equation}
\label{sa1}
\frac{1}{k} \leq h(k) \leq \frac{1}{k-1+1/k}.
\end{equation}

We show (\ref{sa1}) by induction on $k$. The base case $k = r$ is clear since $h(r) = \frac{1}{r-1+1/r}$, . For the induction step, note that by (\ref{sa0}) we have $h(k) = \frac{1 - h(k+1)}{(k-1)}$. By the induction hypothesis, this implies
\[
 \frac{1 - \frac{1}{k + 1/(k+1)}}{k-1}  \leq h(k) \leq \frac{1 - \frac{1}{k+1}}{k-1}. 
\]
 Now observe that 
\[ \qquad \qquad \qquad  \qquad \qquad  \frac{1 - \frac{1}{k + 1/(k+1)}}{k-1} \geq \frac{1}{k}  \quad \text{and} \quad \frac{1 - \frac{1}{k+1}}{k-1}  \leq \frac{1}{k-1+1/k}  \qquad   \text{ for } k \geq 2,\]
 which concludes the proof of property (A2).
 
 Finally, for property (A1), we need to show that $h(k+1) \leq h(k)$; by (\ref{sa0}), it suffices to show that $1 - h(k) (k-1) \leq h(k)$, which holds since $h(k) \geq 1/k$.
\end{proof}

\begin{proposition}
\label{maincor12}
The function $h = h_{\infty}$ satisfies properties (A1)--(A3).
\end{proposition}
\begin{proof}
For any value $k$,  the sum defining $h_{\infty}(k)$ is an alternating series whose terms decrease monotonically,  hence $h_{\infty}(k) = \lim_{r \rightarrow \infty} h_r(k)$. Properties (A1)--(A3) are defined in terms of non-strict inequalities; since each $h_r(k)$ individually satisfies these properties, by continuity their limit satisfies them as well.
\end{proof}

Propositions~\ref{maincor11} and \ref{maincor12}, combined with Theorem~\ref{result2}, imply Corollary~\ref{maincor1} parts (1) and (2).  To show Corollary~\ref{maincor1} part (3), observe that $\tilde h_{\infty}(k) \leq h_{\infty}(k)$ for all $k$. 

\end{document}